\documentclass[12pt, reqno]{amsart}
\usepackage{amsmath, amstext, amsbsy, amssymb, amscd}
\usepackage{amsmath}
\usepackage{amsxtra}
\usepackage{amscd}
\usepackage{amsthm}
\usepackage{amsfonts}
\usepackage{amssymb}
\usepackage{eucal}
\usepackage{color}
\usepackage[all]{xy}
\usepackage[CJKbookmarks=true]{hyperref}

\newtheorem{theorem}{Theorem}[section]
\newtheorem{lemma}[theorem]{Lemma}
\newtheorem{prop}[theorem]{Proposition}

\theoremstyle{definition}

\newtheorem{claim}[theorem]{Claim}
\newtheorem{defn}[theorem]{Definition}

\newtheorem{remark}[theorem]{Remark}

\numberwithin{equation}{section}

\def\char{\text{char}}

\def\End{\text{End}}

\def\GL{{\text{GL}}}

\def\Ker{\text{Ker}}
\def\Ad{\text{Ad}}
\def\Lie{\text{Lie}}

\def\Id{\text{Id}}

\def\resp{{\sl{resp}}}

\def\ve{{\varepsilon}}

\def\bbk{\mathbb K}

\def\bbc{\mathbb C}

\def\co{\mathcal {O}}

\def\b0{{\bar 0}}
\def\b1{{\bar 1}}

\def\ggg{{\mathfrak g}}
\def\hhh{{\mathfrak h}}

\def\tsd{{\textsf{d}}}

\makeatletter

\newcommand{\Rmnum}[1]{\expandafter\@slowromancap\romannumeral #1@}
\makeatother

{\vskip-\lastskip\medskip
  \noindent
  {\em #1.}\enspace
  }%
{\qed\par\medskip
  }

\begin{document}

\title[Normality of nilpotent orbit closures]{Normality of orthogonal and symplectic nilpotent orbit closures in positive characteristic}

\author{Husileng Xiao and Bin Shu }

\address{Department of Mathematics, East China Normal University, Shanghai 200241, China.}\email{hsl1523@163.com}
\address{Department of Mathematics, East China Normal University, Shanghai 200241, China.}\email{bshu@math.ecnu.edu.cn}
\begin{abstract}
 In this note we investigate the normality of closures of  orthogonal and symplectic nilpotent orbits in  positive characteristic. We prove that the closure of such a nilpotent orbit is normal provided that neither type $d$ nor type $e$ minimal irreducible degeneration occurs in the closure, and conversely if the closure is normal, then any type $e$ minimal irreducible degeneration does not occur in it. Here, the minimal irreducible degenerations of a nilpotent orbit are introduced by W. Hesselink in \cite {Hes} (or see \cite{KP2} from which we take Table \ref{table} for the complete list of all minimal irreducible degenerations). Our result is a weak version in positive characteristic of \cite[Theorem 16.2(ii)]{KP2}, one of the main results of \cite{KP2} over complex numbers.
\end{abstract}

\thanks{ This research is partially supported by the NSF of China (No. 11271130; 11201293; 111126062)£¬ and  Shanghai Key Laboratory of PMMP(13dz2260400)}

\maketitle

\section{Preliminaries}
    \subsection{}
    Let $G$ be a reductive algebraic group over an algebraically closed field $\mathbb{K}$ of positive characteristic, and $\ggg=\Lie(G)$. A nilpotent orbit of  $G$ is  an orbit of  a nilpotent element in $\ggg$ under the adjoint action of $G$.   For all classical groups, the nilpotent orbits were parameterized in terms of partitions. The normality of closures of nilpotent orbit of classical group have been studied by several authors. However, there is still an open question to decide the normality of the closures of nilpotent orbits. Our purpose is to investigate such a problem.

\subsection{} In 1979 and 1980s, Kraft-Procesi in \cite{KP1} and \cite{KP2} determined the normality of orbit closures for all complex classical groups  by using smooth equivalent arguments (with few exceptions of the very even orbits in the special orthogonal group $D_{2l}$ remaining, which was completed by Sommers in \cite{Somm}). For other types, A. Broer in \cite{Bro} finished the work on the normality of nilpotent orbit closures (corresponding to two pairwise orthogonal short root) by vanishing result of cohomology of line bundles of flag variety.

In the case of positive-characteristic fields, J. F. Thomsen in \cite{Tho} proved that A. Broer's result holds in good characteristic and decided the normality of some nilpotent orbit closures (corresponding to two pairwise orthogonal short roots).
 Generally, the difficulty of extending Kraft-Procesi method for linear general groups over complex numbers to the case of positive characteristic fields is the failure of the statement over $\bbc$ ``\textit{Assume the reductive algebraic group $G$ acts regulary on an affine variety $V$. If $\pi:V \rightarrow V_{0}$ is a quotient map and $W \subset V$ is a $G$-stable subvariety, then  the restriction of $\pi$ to $W$ is a quotient map onto $\pi(W)$.}" Donkin in \cite{don} overcome this difficulty by means of representation theory. He viewed the coordinate ring $\bbk[V]$ as a $G$-module,  then considered certain module filtration of  $\bbk[V]$ (called good filtrations). 
This enables him 
  to prove that all closures of nilpotent orbits  of general linear groups in positive character  are normal.

\subsection{}
Recently, E. Goldstein in her doctoral thesis \cite{Gol} investigated the normality of the closures of nilpotent orbits of orthogonal and symplectic groups. There
  Goldstein exploited Donkin's method to the  orthogonal and symplectic groups in positive characteristic. She finally obtained Proposition 5.2 of \cite{Gol} which is crucial to the present  paper (see Theorem \ref{proposition5.2}). With aid of   Goldstein's theorem, we are able to  decide the normality of  some  nilpotent orbit closures. 
Let us introduce our main result in the next subsections.

\subsection{} Throughout the paper, we always assume $\bbk$ is an algebraically closed filed of characteristic $p>2$.
Let $V$ be finite dimensional  vector space over $\bbk$,    $G$ be  one of  the algebraic groups $\text{O}(V)$ or $\text{Sp}(V)$
  which is determined by a nondegenerate  form $(\cdot,\cdot)$ with $(u,v)=\varepsilon(v,u)$ where  $\varepsilon\in\{1,-1\}$. Call $V$ a quadratic space of  type $\varepsilon$ (shortly an orthogonal space in case $\varepsilon=1$, a symplectic space  in case $\varepsilon=-1$).
Let $\ggg=\mathfrak{so}(V,\bbk)$ or $\mathfrak{sp}(V,\bbk)$ be the Lie algebra of $G$. Then the nilpotent orbits $\co_{\varepsilon,\sigma}$ under the adjoint action of $G$ in $\mathfrak{g}$   are completely determined by partitions $\sigma$ of $n=\dim(V)$.  The corresponding young diagram is called an $\varepsilon$-diagram. We denote by $|\sigma|$ the size of $\sigma$, which is equal to $\sum_{i=1}^t ir_i$ for $\sigma=[1^{r_{1}}2^{r_{2}}3^{r_{3}}\dots t^{r_t}]$.

There is a well-known  classification  result on nilpotent orbits of $\ggg=\mathfrak{so}(V,\bbk),\mathfrak{sp}(V,\bbk)$.
\begin{lemma} Let $\sigma=[1^{r_{1}}2^{r_{2}}3^{r_{3}}\dots t^{r_t}]$. Then the following statement hold.
\begin{itemize}
\item[(1)]  The partition $\sigma$ is 1-diagram if and only if   $r_{i}$ is even for even $i$.
\item[(2)]  The partition $\sigma$ is -1-diagram if and only if  $r_{i}$ is even for odd $i$.
\end{itemize}
\end{lemma}

\begin{defn}
Let $\eta$ be an $\varepsilon$-diagram.
\begin{itemize}
\item[(1)] An $\varepsilon$-diagram $\sigma$ is called an $\varepsilon$-degeneration of $\eta$ if $|\sigma|=|\eta|$ and $\co_{\ve,\sigma}\in \overline{\co_{\ve,\eta}}$, which is denoted by $\sigma \leq \eta$. This gives an ordering for $\varepsilon$-diagrams.
\item[(2)] An $\ve$-degeneration $\sigma$ of $\eta$ is called minimal if $\sigma\neq\eta$ and there is no $\varepsilon$-diagram $\nu$ such that $\sigma<\nu<\eta$.
Geometrically this means  $\mathcal{O}_{\varepsilon,\sigma}$ is open in complement of $\mathcal{O}_{\varepsilon,\eta} $ in $\overline{\mathcal{O}_{\varepsilon,\eta}}$.
\end{itemize}
\end{defn}

Then we have the following observations on $\ve$-degenerations.
\begin{lemma}\label{cancelinglemma}
Let  $\sigma=(\sigma_1\geq \sigma_2\geq \cdots)$ and $\eta=(\eta_1\geq \eta_2\geq \cdots)$ be two $\varepsilon$-diagrams and $|\sigma|=|\eta|$. Then the following statements hold.
\begin{itemize}
\item[(1)] The inclusion $\mathcal{O}_{\varepsilon,\sigma}\subset     \overline{\mathcal{O}_{\varepsilon,\eta}}$ happens if and only if
$\sum_{i=1}^{j}\sigma_{i}\leq\sum_{i=1}^{j}\eta_{i}$ for all $j$.
\item[(2)] Assume that $\sigma\leq \eta$ be an $\varepsilon$-degeneration with the first  $r$ rows and the first $s$ columns of both $\sigma$ and $\eta$ coinciding respectively. Denote by $\sigma',\eta'$ the new diagrams obtained by erasing these $r$ rows and $s$ columns from $\sigma$ and $\eta$ respectively,
and put $\varepsilon'=(-1)^{s}$. Then $\sigma'<\eta'$ is an $\varepsilon'$-degeneration and $\mbox{codim}_{\overline{\co_{\ve,\eta}}}\co_{\ve,\sigma}
=\mbox{codim}_{\overline{\co_{\ve',\eta'}}}\co_{\ve',\sigma'}$.
\end{itemize}
\end{lemma}
\begin{proof} For (1), one can be refereed to \cite{Jan2}. For (2), it can be proved  by the same arguments as the complex number case  given in \cite{Gol}.
\end{proof}

In the setup of the above lemma (2), we say that the $\varepsilon$-degeneration $\sigma\leq \eta$ is obtained from the $\varepsilon'$-degeneration $\sigma'<\eta'$
by adding rows and columns.  An $\varepsilon$-degeneration $\sigma\leq \eta$ is called \textit{irreducible} if there is no pair $\sigma'\leq \eta'$ with $| \sigma' |= | \eta' |  < |\sigma | = | \eta | $ such that $(\sigma' ,\eta')$  is obtained from $(\sigma ,\eta)$  by procedure of Lemma \ref{cancelinglemma}(2). In \cite{Hes}, all minimal irreducible degenerations are classified  (see Table \ref{table}).

\subsection{The main result of the present paper}
\begin{theorem} \label{mainthm} Let $V$ be a vector space over an algebraically closed field $\bbk$ of $\char(\bbk)\neq 2$, $G=\mbox{O}(V)$ ({\sl{resp}}. $\mbox{Sp}(V)$) be the orthogonal  ({\sl{resp}}. symplectic)  group corresponding to the defining non-degenerate bilinear form  $(\cdot,\cdot)$ of type $\varepsilon=1$  ({\sl{resp}}. $\varepsilon=-1$). Then the following statements hold.
\begin{itemize}
\item[(1)] For any nilpotent orbit $\co_{\ve,\eta}$, its closure  $\overline{\co_{\ve,\eta}}$ is normal if  $\eta$ has neither degeneration of type $d$ nor degeneration of type $e$.
\item[(2)] Conversely, for a given nilpotent orbit $\co_{\ve,\eta}$, if its closure $\overline{\co_{\ve,\eta}}$ is normal then $\eta$ does not contain any degeneration of type $e$.
\end{itemize}
Here, the  types $d$ and $e$ are  listed as in  Table \ref{table}.
\end{theorem}

The main text of the paper will be devoted to the proof of the above theorem. In the concluding subsection, some examples are presented for demonstrating the main theorem.   Our approach is based on Theorem \ref{proposition5.2} by Goldstein (cf. \cite{Gol}) along with some crucial observations on the separability and the decomposability  arising from quadratic spaces under action of reductive groups (see Lemma \ref{forthcominglemma1} and  Claim \ref{claim}), which make Kraft-Procesi's arguments on smoothly equivalent singularities over complex numbers  in \cite{KP2} revivified in our case.

For the special orthogonal and exceptional groups, we will give the investigation on the normality of their nilpotent orbit closures elsewhere.

\begin{table}\caption{Classification of minimal irreducible degenerations}\label{table}

\begin{tabular}{|c|c|c|c|c|}  \hline
type & $a$  \  \  \ & $b$ & $c$ & $d$ \\
\hline
Lie algebra  & $\mathfrak{sp}_{2}$ & $\mathfrak{sp}_{2n}$ & $\mathfrak{so}_{2n+1}$ & $\mathfrak{sp}_{4n+2}$\\
\hline
$\varepsilon$ & $-1$ & $ -1 $ & $ 1 $& $-1$  \\
\hline
$\eta $   &  $(2)$  & $(2n)$ & $(2n+1)$ & $(2n+1,2n+1)$ \\
\hline
$\sigma $ &  $(1,1)$ & $(2n-2,2)$ & $(2n-1,1,1)$ & $(2n ,2n ,2)$   \\
\hline
$\mbox{codim}_{\overline{\co_{\ve,\eta}}}\co_{\ve,\sigma}$ & 2 & 2 & 2 & 2 \\
\hline
\end{tabular}
\\
\begin{tabular}{|c|c|c|c|}
\hline
$e$ & $f$  &  $g$   &   $h$ \\
\hline
$\mathfrak{so}_{4n}$ &  $\mathfrak{so}_{2n+1}$ & $\mathfrak{sp}_{2n}$ & $\mathfrak{so}_{2n}$ \\
\hline
1 & 1 & -1 & 1 \\
\hline
$(2n,2n)$ &  $(2,2,1^{2n-3})$ & $(2,1^{2n-2})$ & $(2,2,1^{2n-4})$\\
\hline
$(2n-1,2n-1,1,1)$ & $(1^{2n+1})$ & $1^{2n}$ & $1^{2n}$ \\
\hline
 $2$ & $(4n-2)$ & $(2n)$ & $(4n-2)$ \\
 \hline
\end{tabular}\\
\end{table}

\section{On  smooth property of two canonical maps}
\subsection{} Let us first recall the main result Proposition 5.2 in \cite{Gol}  which is actually a modular version of \cite[Theorem 9.2(ii)]{KP2}, one of the main results of \cite{KP2} over complex numbers.
This result will be important  to the proof of Theorem \ref{mainthm}.
\begin{theorem}(Goldstein \cite{Gol},proposition 5.2) \label{proposition5.2}
Let $\mathcal{O}$ be a nilpotent orbit of the symplectic or orthogonal group. Then
 $\overline{\mathcal{O}}$ is normal if and only if  it is normal at all points contained in the orbits $\mathcal{O}_{i} \subset \overline{\mathcal{O}}$ of codimension 2.
\end{theorem}

We then have some consequences.
 \begin{lemma}\label{lemma2.2}
 Assume that the pair $(\eta,\sigma)$ is of type $a,b,c$  as in Table \ref{table}. Then $\overline{\mathcal{O}_{\varepsilon,\eta}}$ is normal .
 In particular ,$\overline{\mathcal{O}_{\varepsilon,\eta}}$ is normal at  $\mathcal{O}_{\varepsilon,\sigma}$
 \end{lemma}
 \begin{proof} It is well known that $\overline{\mathcal{O}_{\varepsilon,\eta}}=\mathcal{N}$, the nilpotent cone of the corresponding Lie agebra,  which is normal (cf. \cite[Corollary 8.5]{Jan2}).
 \end{proof}
\begin{lemma}\label{lemma2.3} The closure $\overline{\mathcal{O}_{\varepsilon,\eta}}$ is not  normal at  $\mathcal{O}_{\varepsilon,\sigma}$ in the case of type
  $e$ in Table \ref{table}.
\end{lemma}
\begin{proof} The closure
 $\overline{\mathcal{O}_{\varepsilon,\eta}}$  is reducible, which is  a non-trivial union of two closures of equal-dimentional  $\mbox{SO}(V)$-orbits. So it is not normal. Note that $\co_{\ve,\sigma}$ is of codimension $2$ in $\overline{\co_{\ve,\eta}}$.
By Theorem \ref{proposition5.2},  $\overline{\co_{\varepsilon,\eta}}$ is not  normal at  $\co_{\ve,\sigma}$.
\end{proof}

\begin{remark}\label{lemma2.3Rem} (1) For the above Lemma, we make another account,  
as a direct consequence of
the following general fact.

Let $X$ be a reducible  affine variety which is the union $Y \cup Z$ of distinct irreducible  components $Y$ and $Z$. Then for all $x \in Y \cap Z $,
$X$ is not normal at $x$.
This comes from a fact that for the local ring $A=O_{X,x}$, irreducible components of $\mbox{Spec}(A)$ are in 1-1 correspondence with irreducible components of $X$ containing $x$ (see  \cite[Proposition 2.4.12]{Liu}).

(2) The closure
$\overline{\mathcal{O}_{\varepsilon,\eta}}$ is normal at  $\mathcal{O}_{\varepsilon,\sigma}$ in the case of  type $f,g,h$,  which comes from the fact that it is the closure of the orbit of a highest weight vector in the adjoint representation (cf. \cite[Proposition 8.13]{Jan2} and the remark following it). We are only concerned with those codimension $2$ orbits inside. So we will not use this fact in this paper .
\end{remark}
\subsection{The canonical maps $\pi$ and $\rho$}\label{canonicalmaps}
Let $V$ and $U$ be two quadratic spaces of type $\varepsilon$ and $-\varepsilon$ respectively  with $\dim V:=n\geq\dim U:=m$, and $G(U)$ and $G(V)$ the orthogonal or sympletic groups defined by the given quadratic forms  on $U$ and $V$ respectively, depending on the values of  $\ve$ and of $-\ve$. Denote by $\ggg(U)$¡¡and $\ggg(V)$ the Lie algebras of $G(U)$ and of $G(V)$ respectively. Denote by  $L(V,U)$ the linear space $\bbk$-spanned by all linear maps between $V$ and $U$.
For a given $X \in L(V,U)$,  define $X^*$ to be the adjoint map of $X$, this is to say, the unique element in
$L(U,V)$  satisfying  $(Xv,u)=(v,X^*u)$ for all $v \in V ,u \in U$. We consider the following canonical maps:
$$\xymatrix{
  L(V,U) \ar[d]_{\rho} \ar[r]^{\pi} &    \mathfrak{g} (U)   \\
\mathfrak{g}(V) }$$
where $\pi(X)=X\circ X^*$, $\rho(X)=X^*\circ X$.
 Define a natural  action of the group $G(U)\times G(V)$ on $L(V,U)$  via $(g,h)X=gXh^{-1}$. Then First theorem of classical invariant theory  says
   $\pi$ and $\rho$ are  $G(U)$- and $ G(V)$-equivariant  quotient maps, in positive character  case see (\cite[Theorem 2.7]{Gol}).

   Let $L'(V,U):=\{Y \in  L(V,U) \mid Y \mbox{ is surjective} \}$, simply written as $L'$.
We immediately have the following  lemma by the same arguments as in the  proof of \cite[Lemma 4.2]{KP2}.
\begin{lemma}(\cite[Lemma 2.9]{Gol})\label{fiber lemma}
For any $Y \in L'$ the stabilizer of $Y$ in $G(U)$ is trivial and $\rho^{-1}(\rho(Y))$ is an orbit under $G(U)$.
\end{lemma}

\subsection{} Now for a given quadratic space $V$, take a nilpotent element $D\in \ggg(V)$ with
conjugacy class $\co_{\ve,\eta}$. Consider the new form on $V$ given by $|v,u|:=(u,Dv)$.
Clearly it is of type $-\ve$ and its kernel is exactly $\Ker D$. Set $U=\text{Im} D$, and take $X:V\rightarrow U$ which is defined by the canonical decomposition $D=I\circ X:V\rightarrow U\hookrightarrow V$, where $I:U\hookrightarrow V$ is the inclusion map.  Then we can canonically
define a non-degenerate form on
$U$ of type $-\ve$. Note that $X^*=I$. We have $D=I\circ X=X^*\circ X\in \mathfrak{g}(V)$. Then $D':=X\circ X^{*}$ is also a nilpotent element in $\mathfrak{g}(U)$ with  conjugacy class $\co_{\ve',\eta'}$, where the corresponding young diagram $\eta'$ of  $D'$ is obtained  from the young diagram  of $\eta$ by erasing the first column (see \cite[\S2.2 and \S2.3]{KP1} and \cite[\S4.1]{KP2}).
 Set $N_{\varepsilon,\eta}:=\pi^{-1}(\overline{\mathcal{O}_{\ve',\eta'}})$. Note that $ N_{\varepsilon,\eta}$ is stable under $G(U)\times G(V)$. We then have the following observation by the same arguments as in the proof of \cite[Lemma 4.3]{KP2}.
\begin{lemma}(\cite[Lemma 2.8]{Gol}) \label{lemma2.6}  The following statements hold.
\begin{itemize}
\item[(1)] $\rho(N_{\varepsilon,\eta})=\overline{\mathcal{O}_{\varepsilon,\eta}}$
\item[(2)] $\rho^{-1}(\mathcal{O}_{\varepsilon,\eta})$ is a single orbit under $G(U)\times G(V)$ contained in $N_{\varepsilon,\eta}\cap L'$
\item[(3)]$\pi(\rho^{-1}(\mathcal{O}_{\varepsilon,\eta}))=\mathcal{O}_{-\varepsilon,\eta'}$
\item[(4)]  The above three statements are still valid to other degenerations of $\eta$ by erasing the first column.
\end{itemize}
\end{lemma}

\subsection{} Keep the setup in \S\ref{canonicalmaps}.
We have the following proposition which is somewhat a weaker version  in positive characteristic of  \cite[Proposition 11.1]{KP2}. This proposition will play a role similar to \cite[Proposition 11.1]{KP2} and the followed remark in the complex number case.
\begin{prop} \label{prop2.7} Maintain the notations as above. And recall $n=\dim V\geq\dim U=m$. Then the following statements hold.¡¡¡¡
\begin{itemize}
\item[(1)] The map $\pi$ is smooth in $L'$,  and $\pi(L') =\{ D\in \ggg(U)\mid \text{Rank}(D)\geq 2m-n\}$.
\item[(2)] $\rho (L')=\{D \in \ggg(V) \mid \text{Rank}(D)=m\}$ and $\rho|_{L'}: L'\rightarrow \rho(L')$ is a smooth morphism, admitting the fibers isomorphic to $G(U)$

\end{itemize}

\end{prop}

\begin{proof} The description here on the images of $\pi$ and $\rho$ is the same thing as what has been done in the case of characteristic $0$ given in \cite[Proposotion 11.1]{KP2}. So we only need prove the remaining assertions.

(1)  In the proof of \cite[Proposotion 11.1(1)]{KP2}, the authors calculate explicitly  the tangent map $(\textsf{d}\pi)_X$ of $\pi$ at any $ X \in L'(V,U)$,  showing that $(\textsf{d}\pi)_X$ is surjective.
Their arguments are still valid for the algebraically closed field $\bbk$ (note that $\char(\bbk)> 2$). This is to say, in our case we still have that the tangent map $(\textsf{d}\pi)_X$ is surjective for all $ X \in L'(V,U)$. By \cite[Proposition 10.4]{Har}
$\pi$ is smooth at all points in $L'(U,V)$. We complete the proof of (1).

(2) First notice that  $L'=L'(V,U)$ is one  orbit under $\GL(V)$ by right multiplication and so is $\rho (L')$ under action $D \mapsto (g^{-1})^*\circ D\circ g^{-1} $. And $\rho$ is a $\GL(V)$-equivariant morphism. Fix  $X\in L'$ and let $q$ and $q'$ be two orbit mappings associated with $X$ as below
$$\xymatrix{
  \GL(V) \ar[d]_{q'} \ar[r]^{q} & \GL(V).X\\
     {\GL(V).\rho (X)}}$$
We first  prove $L'\simeq \GL(V)/H$, $\rho (L')\simeq \GL(V)/H'$. Here $H$ (\resp. $H'$) is the centralizer of $X\in L'$ (\resp. $\rho(X) \in \mathfrak{g}(U)$). So by \cite[12.4]{Hum} we only need to prove that both $q$ and $q'$ are separable. For this, by \cite[5.5]{Hum} we only need to prove that $q$ and $q'$ are smooth.  In fact,  $L'$ is  an open subset of the vector space  $L(V,U)$ and  $\mathcal{T}_{X}(L')=L(V,U)$. Then
 we have $\text{Im}((\tsd  q)_{\Id})=X\cdot \mathfrak{gl}(V)$. Here the dot action "$\cdot$" is right multiplication (composition), and "$\Id$" stands for the identity element. On the other side, the surjective property of $X$ implies that $X\cdot \mathfrak{gl}(V)=L(V,U)$.
 Hence $(\tsd q)_\Id$ is surjective. Note that $q$ is an orbit map, thereby $\GL(V)$-equivariant.
 By \cite[Proposition 10.4]{Har},
 $q$ is smooth.

Next we prove $q'$ is smooth. As  $\rho (L')$ is an orbit of $\GL(V)$, by the same arguments as in the last paragraph it suffices for us to prove $(\tsd q')_\Id$ is surjective at the identity element.
We have $(\tsd q')_\Id(Z)=-Z^{*}D-DZ$ with $D=\rho(X) \in \rho(L')$.
 We will further show that $\dim\text{Im}((\tsd q')_\Id)=\dim L'-\dim G(U)=\dim \rho(L')$ in the forthcoming Lemma \ref{forthcominglemma1} (the second equality is due to the fact that  $\rho^{-1}(D)$ is isomorphic to $G(U)$). Thus, $\tsd q'$ is surjective at $\Id$, thereby surjective at  all  points in $\GL(V)$. So  $q'$ is smooth.
 Thus we have $L'\simeq \GL(V)/H$ and $\rho (L')\simeq \GL(V)/H'$. Note that under the  above identification,  $q'=\rho|_{L'}\circ q$. Hence $\rho|_{L'}$ is smooth. Therefore the proof is  completed  modulo Lemma \ref{forthcominglemma1}.
\end{proof}

\begin{lemma}\label{forthcominglemma1}
Keep the notations and assumptions in the above proposition. Then
$\dim \text{Im} (\tsd q')_{\Id}=\dim(L')-\dim(G(U))=\dim(\rho(L'))$.
\end{lemma}
\begin{proof} We only need to prove the lemma in the case of $\mathfrak{g}(V)=\mathfrak{sp}(V)$, and omit the arguments for the case of $\mathfrak{g}(V)=\mathfrak{o}(V)=\mathfrak{so}(V)$ because the latter is the same. So we assume $\dim(V)=2l\geq 4$ (it is trivial when  $\dim(V)=2$), and further take $\{e_{1},\cdots ,e_{l}$,$e_{1+l},\cdots, e_{2l}\}$, a basis of $V$ compatible with the quadratic form $(\cdot,\cdot)$. This is to say $(e_{i},e_{j})=\mbox{sgn}(i-j)\delta_{l,\mid i-j \mid}$.
 Then by a direct computation we have
 $$
E_{i,j}^{*}=\begin{cases}
E_{j+l,i+l},  &\mbox{for } i,j\leq l; \\
E_{j-l,i-l}& \mbox{for }l < i,j\leq 2l; \\
E_{j-l,i+l}& \mbox{for } i\leq l<j; \\
E_{j+l,i-l}& \mbox{for }j\leq l <i.
\end{cases} $$
Now we view  $(\tsd q')_\Id$ (composing with  $\mathfrak{g}(V)\hookrightarrow \mathfrak{gl}(V)$) as an  element of $\End(\mathfrak{gl}(V))$. We proceed the arguments by different cases.

(1) Suppose $m$ is even, say $2k$. By description of $\rho(L')$ which is an orbit in the above proposition,  we may take $D$ to be a diagonal matrix of rank $m$ as below
$$D=\mbox{diag} \{1,\dots 1,0 ,\dots ,0,-1,\dots -1,0 ,\dots ,0 \},$$
where $k$-times $1$ and $k$-times $-1$ foremost and continuously occur  in the first $l$-block and the last $l$-block respectively. It is elementary to check that subspace  $V(i,j):=\bbk\mbox{-span}\{ E_{i,j}, E_{i,j}^{*}\}$  are stable under the map $(\tsd q')_\Id(Z)=-Z^{*}D-DZ$. And
$$\mathfrak{gl}(V)=\bigoplus_{i,j\leq l}V(i,j) \oplus \bigoplus_{i\leq l ,l<j,j-l\leq i}V(i,j)\oplus \bigoplus_{j\leq l ,l<i,i-l\geq j}V(i,j).$$
We find by explicit calculation  that the dimension of image of restriction of  $(\tsd q')_\Id$ in each $V(i,j)$ are independent of $\mbox{char}(\mathbb{K})$ (under assumption $\mbox{char}(\bbk) > 2$ ). Note that the statement of Lemma holds in the case of complex numbers. 
So the lemma follows in this case.

(2) Suppose $m$ is odd, say $2k+1$. By the same reason as before we may take $D$  to be a  matrix of rank $m$ as below
$$D=\mbox{diag} \{1,\dots 1,0 ,\dots ,0,-1,\dots -1,0 ,\dots ,0 \}+E_{1,l+1},$$
where for the diagonal matrix,  $k$-times $1$ and $k$-times $-1$ foremost and continuously occur  in the first $l$-block and the last $l$-block respectively.
 We construct $V(i,j)$ (it is enough to do for some pairs of $(i,j)$) as below
$$V(i,j) =\begin{cases}
\bbk\mbox{-span}\{ E_{i,j}, E_{i,j}^{*}\}, \mbox{ for }(i,j)\neq (1,k) ,(l+1,k),(k,1)£¬(k,l+1), \forall k \leq 2l;\\
\bbk\mbox{-span}\{ E_{i,j},E_{j-l,1},E_{1,j},E_{j-l,l+1}\},\mbox{ for }i=l+1,l+1<j<2l;  \\
\bbk\mbox{-span}\{ E_{l+1,j},E_{1,j},E_{l+j,l+1},E_{l+j,1}\},\mbox{ for }  i=l+1,,1<j<l; \\
\bbk\mbox{-span}\{ E_{1,1}, E_{1,l+1},E_{l+1,1},E_{l+1,l+1}\}, \mbox{ for } (i,j) =(1,1).
\end{cases}$$
 It is easy to check that those subspaces of $\mathfrak{gl}(V)$ are stable under $(\tsd q')_{\Id}$ and $\mathfrak{gl}(V)$  can be written as a
direct sum of some of those $V(i,j)$'s. By the  same argument above we can prove our lemma in this case. So we complete the proof of lemma.
\end{proof}

\begin{lemma}\label{locally closed lemma}
 For any locally closed $G(U)$-stable subset $W \subseteq L'$ the image $\rho(W)$ is also locally closed and $\rho|_{W}:W \rightarrow \rho(W)  $ is smooth. Here by the term ``locally closed" it means either in the Zariski topology or in the \'{e}tale topology.
\end{lemma}

  In the case of complex numbers  and the \'{e}tale topology, the above  lemma  is a direct consequence of the  fact that $\rho|_{L'}$ is local trivial fibration in the \'{e}tale topology (see remark after \cite[Proposition 11.1]{KP1}).  In the case of positive-characteristic fields, we give another  proof with aid of the following lemma.

\begin{lemma}\label{top lemm}
Let $f:X \rightarrow Y$ be a continuous, surjective  and open map of topology spaces. Assume that $W \subseteq X$ is locally closed, and the subset  $W$ and  its closure $\overline{W}$ are full,  then $f(W)$ is also locally closed. (Call a subset $W \subseteq X$  full provided that $W=f^{-1}(f(W))$.)
\end{lemma}

\begin{proof}
Since $f$ is open and $\overline{W}$ ({\sl{resp.}} $\overline{W}-W$) is closed and full, so $f(\overline{W})$ ({\sl{resp.}} $f(\overline{W}-W)$) is closed.
This implies  $f(W)=f(\overline{W})-f(\overline{W}-W)$ is open in its closure $\overline{f(W)}=f(\overline{W})$.
\end{proof}

\begin{proof}\textit{(of Lemma \ref{locally closed lemma})} Since smoothness is preserved under base changes(cf. \cite[Proposition 10.1 in Ch.III]{Har}, or \cite[(4.9-4.10)]{GW}), the second statement follows from
Proposition \ref{prop2.7}(2) and the first statement. So it is enough to prove the first statement.

 Keep in mind the notations appearing in the proof of Proposition \ref{prop2.7}, where we  have $L'\simeq \GL(V)/H$ and $\rho (L')\simeq \GL(V)/H'$. Under the identification
we have  $q'=\rho|_{L'}\circ q$. Note that $q'$ is open map by \cite[12.1(2)]{Hum}, hence $\rho|_{L'}$ is also open map . Since $W$ is $G(U)$-stable, so is its closure $\overline{W}$ in $L'$. Hence both $W$ and $\overline{W}$  are  full by Lemma  \ref{fiber lemma}. Combining with Lemma
\ref{top lemm}, we complete the proof.
\end{proof}

\section{smoothly equivalent singularities}

\subsection{} Maintain the notations as previously. Let us first recall some general theory.
\begin{defn} \label{def3.1} Let $X, Y$ be two varieties, and $x \in X$, $ y \in Y$. We call singularity of $X$ at $x$  \textit{smoothly equivalent} to  singularity of $Y$ at $y$ if there exist a pair of $(Z,z)$ with $Z$ a variety containing a point $z$ and  two morphisms $\varphi,\psi$:
$$\xymatrix{
  Z \ar[d]_{\psi} \ar[r]^{\varphi} &     X   \\
  Y                     }$$
with $\varphi(z)=x,\psi(z)=y$ and $\varphi,\psi$ are smooth at $z$. This defines an equivalent relation between pointed varieties and denote the equivalence class of $(X,x)$ by $\text{Sing}(X,x)$.
\end{defn}
In this paper, for a subset $X' \subset X $ such that   $\text{Sing}(X,x')=\text{Sing}(X,y')$ holds for all $x',y' \in X'$,the notation $\text{Sing}(X,X')$ is short-hand for  $\text{Sing}(X,x')$ where $x'$ is any point in  $X'$.

\begin{lemma} \label{smooth lemma}
Assume $\mbox{Sing}(X,x)=\mbox{Sing}(Y, y)$. Then $X$ is normal at $x$ if and only if $Y$ is normal at $y$.
\end{lemma}
\begin{proof}
Let $\varphi:Z \rightarrow X $ be a smooth map. For a point $z \in Z$, write $\varphi(z)=x$. By \cite[3.24(b) in Ch.1]{Mil} there are affine neighbourhood $V$ of $z$
and $U$ of $x$ and an \'{e}tale map $g:V \rightarrow V'$¡¡
¡¡¡¡¡¡(where  $V'=\mathbb{K}_{U}^{n}$ is the affine $n$-space over $U$) such that
restriction of $\varphi$  to $U$ is $\varphi|_{U}=p \circ g$,¡¡
¡¡¡¡where $p:V' \rightarrow U$ is the natural projection.

Since an \'{e}tale map preserves normality, we have $V$ is normal at $z$ if and only if $V'$ is normal at $g(z)$ by \cite[3.17(b) in Ch.1]{Mil}.
We also have $V'$  is normal at $g(z)$ if and only if $U$ is normal at $x$ by the fact that a commutative ring $A$ is normal if and only if
the polynomial ring $A[T]$ is normal.
\end{proof}

\subsection{}  In spirit of  the arguments in \cite{KP2} for the complex classical groups, we will manage to make the propositions there of  ``induction by canceling rows" and of ``induction by canceling columns" revived in our case.

\begin{theorem}\label{3.3}
Assume that the $\varepsilon$-degeneration $\sigma\leq \eta$ is obtained from the $\varepsilon'$-degeneration $\sigma'<\eta'$ by adding rows and columns. Then
$$\mbox{Sing}(\overline{\mathcal{O}_{\varepsilon,\eta}},\mathcal{O}_{\varepsilon,\sigma})
=\mbox{Sing}(\overline{\co_{\ve',\eta'}},\co_{\ve',\sigma'}).$$
\end{theorem}
For this, we will make some necessary preparation. Theorem \ref{3.3} will directly follow from the forthcoming Propositions \ref{propcancellingrows} and \ref{prop3.8}.

\begin{prop} \label{propcancellingrows} (Induction by canceling rows)
Assume that the $\varepsilon$-degeneration $\sigma\leq \eta$ is obtained from the $\ve$-degeneration $\sigma'<\eta'$ by adding rows,
then
$$\mbox{Sing}(\overline{\co_{\ve,\eta}},\co_{\ve,\sigma})=\mbox{Sing}
(\overline{\co_{\ve',\eta'}},\co_{\ve',\sigma'}).$$
\end{prop}
This proposition corresponds to \cite[Proposition 13.4]{KP2}. The proof of \cite[Proposition 13.4]{KP2} is dependent on \cite[Proposition 13.1]{KP2}. However, the arguments for the latter can not carry out in positive characteristic because it seems impossible to find $H'$-stable decomposition used there for arbitrary reductive group. Fortunately, we observe that some desirable decompositions still exist for the groups we are concerned with (we will write out it explicitly, see Lemma \ref{forthcoming3.6}).  Let us first recall the notion of cross sections (see \cite[12.4]{KP2}).
\begin{defn}
Let $X$ be a variety with a regular action of an algebraic group $G$. A \textit{cross section }at a point $x \in X $ is defined to be a locally closed subvariety
$S\subseteq X$ such that $x\in S$ and the map $G \times S \longrightarrow X, (g,s)\mapsto g\cdot s$, is smooth at the point $(e,x)$. Of course we have $\text{Sing}(S,x)=\text{Sing}(X,x)$.
\end{defn}

\noindent \textsc{Proof of Proposition \ref{propcancellingrows}}:  (Here we just replace the notations $D,D',E,E',F$ in the proof of  \cite[ Proposition 13.4]{KP2} by $x,x',y,y',z$ respectively)  Let $V$ be a quadratic space of type $\varepsilon$ of dimension $|\eta|$, and $x \in \mathcal{O}_{\varepsilon,\eta}\subset \mathfrak{g}(V)$. By assumption the diagrams $\eta$ and $\sigma$ are decomposed into $\eta=\nu+\eta'$, $\sigma=\nu+\sigma'$ respectively where $\nu$ is a common row
of $(\sigma,\eta)$ and  $(\sigma,\eta)$ is obtained from $(\sigma',\eta')$ by adding the diagram  $\nu$. Under the decomposition  $V=W\oplus V'$, write $x=(z,x')\in \mathfrak{g}(W)\oplus \mathfrak{g}(V')$, $y=(z,y')\in \mathfrak{g}(W)\oplus \mathfrak{g}(V')$
where $y\in  \mathcal{O}_{\varepsilon,\sigma}$; $x' \in \mathcal{O}_{\varepsilon,\eta'} \subseteq  \mathfrak{g}(V')$,
$y' \in  \mathcal{O}_{\varepsilon,\sigma'}\subseteq  \mathfrak{g}(V')$ and $z \in \mathfrak{g}(W)$. Now let us consider the following reductive groups $G:=\GL(V)$,
$G':=\GL(W)\times \GL(V')$, $H:=G(V)$, $H'=G(W)\times G(V')$.  Set $\ggg=\Lie(G)$, $\ggg'=\Lie(G')$, $\hhh=\Lie(H)$ and $\hhh'=\Lie(H)$.
 The assumptions of \cite[Proposition 13.1]{KP2} hold in this context,
namely
\begin{itemize}
\item[(1)] $\text{codim}_{\overline{G'x}}G'y=\text{codim}_{\overline{Gx}}Gy$.
\item[(2)] $\overline{G'x}\cap \mathfrak{h}'=\overline{H'x}$.
\item[(3)] $\overline{Gx}$ is normal at $y$.
\end{itemize}
The above (1) is due to the property of nilpotent orbits in the type $A$. As to (3), it follows from the main result of \cite{don}.
By the same arguments as in the proof of \cite[Proposition 13.4]{KP2}, we only need to prove that the statement of \cite[Proposition 13.1]{KP2} is still valid to our situation. This is to say, we only need to prove that $\text{Sing}(\overline{Hx},y)=\text{Sing}(\overline{H'x},y)$, which will result from Lemma
\ref{forthcoming3.6} presented below.
 The proof is completed modulo the forthcoming lemma.

\begin{lemma} \label{forthcoming3.6} Maintain the notations as in the above proof. 
Then $\text{Sing}(\overline{Hx},y)=\text{Sing}(\overline{H'x},y)$.
\end{lemma}

\begin{proof}
Firstly let us make the following claim.

 \begin{claim}\label{claim} There is an  $H'$-stable decomposition of $\ggg$ as below
$$\mathfrak{g}=\mathfrak{h'}\oplus M'\oplus M_{0}\oplus D$$
where $\mathfrak{g'}=\mathfrak{h'}\oplus M'$, $\mathfrak{h}=\mathfrak{h'}\oplus M_{0}$.
\end{claim}

We certify the claim. Suppose that $(\cdot,\cdot)$ is the defining quadratic form of $G$, i.e. $G:=\{ g\in \GL(V)\mid(gv,gu)=(u,v)\}$. For a given $f\in \End(V)$, define
   $f^{*}$ via $(fu,v)=(u,f^{*}v)$. Then we have $H=\{g\in \GL(V)\mid g^{-1}=g^{*}\}$, and
   $\mathfrak{h}=\{x\in \mathfrak{gl}(V)\mid -x=x^{*}\}$. We further have $\End(V)=\mathfrak{g}\bigoplus M$, where $M=\{x\in \End(V)\mid x=x^{*}\}$.
Set    $M'=\{x\in g'\mid x=x^{*}\}$ and $M_{0}=\{x \in \mathfrak{h}\mid x(v)\in W \text{ for all }v \in V';x(w)\in  V' \text{ for all }w \in W \}$,
   $D=\{ x \in \mathfrak{g}\mid x(v)\in W \text{ for all }v \in V'; x(w)\in  V' \text{ for all }w \in W ; x=x^{*}\}.$
It is  easy to verified that those spaces are $H'$-stable and $$\mathfrak{g}=\mathfrak{h'}\oplus M'\oplus M_{0}\oplus D.$$
So the claim follows.

Secondly, by the same argument as in the proof of \cite[Proposition 13.1]{KP1}, we can complete the proof of the lemma.
For the readers' convenience, we  give the detailed arguments.

 Now we have vector space decomposition
$$ \mathfrak{h}'=[\mathfrak{h}',y]\oplus N'_{0},\;\;M'=[M',y]\oplus\overline{N'}$$
and $$M_{0}=[M_{0},y]\oplus\overline{N_{0}},\;\;D=[D,y]\oplus \overline{D}.$$
Set $N:=N'_{0} \oplus \overline{N'}\oplus \overline{N_{0}}\oplus\overline{D}$.
Then we have
\begin{align}\label{decompos}
\mathfrak{g}&=[\mathfrak{g},y]\oplus N\cr
\mathfrak{g'}&=[\mathfrak{g'},y]\oplus N',\mbox{where }N':=N \bigcap \mathfrak{g'},\cr
\mathfrak{h}&=[\mathfrak{h},y]\oplus N_{0},\mbox{where }N_{0}:=N\bigcap\mathfrak{h},\cr
\mathfrak{h'}&=[\mathfrak{h'},y]\oplus N_{0}',\mbox{where }N_{0}':=N\cap\mathfrak{h}'.
\end{align}
So we can define
$$S:=(N+y)\cap\overline{Gx},$$
$$S'=(N'+y)\cap\overline{G'x},$$
$$S_{0}=(N_{0}+y)\cap\overline{Hx},$$
$$S_{0}'=(N_{0}'+y)\cap\overline{H'x}.$$
From the forthcoming Lemma \ref{forthcominglemma3.7}   it follows that those $S,S',S_{0},S_{0}'$ are cross sections of $\overline{Gx},\overline{G'x},\overline{Hx},\overline{H'x}$ at $y$ respectively under the adjoin action of the corresponding groups.
  So we have $\text{Sing}(\overline{H'x},y)=\text{Sing}(S_{0}',y)$, $\text{Sing}(\overline{Hx},y)=\text{Sing}(S_{0},y)$. Hence it is  enough to prove $\text{Sing}(S_{0}',y)=\text{Sing}(S_{0},y)$.
By the same arguments as  in \cite[Proposition 13.4]{KP2},  we have $S_{0}'=S'\cap\mathfrak{h}$.
From  Property (1) listed in the  proof of Proposition \ref{propcancellingrows} and  the forthcoming Lemma \ref{forthcominglemma3.7}, we have $\dim_{y}S=\dim_{y}S'$. 
From Property (3) appearing as previously,
$S$ is normal in $y$. So there  exists an open neighborhood $F$ of $y$ in $S$ such that
$F\subset S'$ and $F$ open in  $ S'$. From the known fact
$$ S\cap\mathfrak{h}\supseteq S_{0}\supseteq S_{0}'=S'\cap\mathfrak{h},$$
we have that $F\cap \mathfrak{h}$ is a common open neighborhood of $y$ in $ S_{0}$ and $S_{0}'$.
Thus we have $\text{Sing}(S_{0}',y)=\text{Sing}(S_{0},y)$. So the proof is completed, modulo the following Lemma \ref{forthcominglemma3.7}.
\end{proof}

\begin{lemma}\label{forthcominglemma3.7} Maintain the notations and assumptions as above. Then
  $S,S',S_{0},S_{0}'$ are cross sections of $\overline{Gx},\overline{G'x},\overline{Hx},\overline{H'x}$ at $y$ respectively under the adjoint action of
corresponding groups. Furthermore,
$\mbox{codim}_{\overline{Gx}}Gy=\mbox{dim}_{y}S$,
$\mbox{codim}_{\overline{Gx'}}Gy'=\mbox{dim}_{y'}S'$,
$\mbox{codim}_{\overline{Hx}}Hy=\mbox{dim}_{y}S_{0}$ and
$\mbox{codim}_{\overline{H'x}}H'y=\mbox{dim}_{y}S'_{0}$
\end{lemma}
\begin{proof} We only need to prove the claim in the case of $S_{0}'$ at $y$. The remaining things are the same. In fact we already have $\mathfrak{h'}=[\mathfrak{h'},y]\oplus N_{0}'$ with $N_{0}':=N\cap\mathfrak{h}'$. Consider the following commutative diagram
$$
\xymatrix{
 \mbox{Ad}^{-1}(\overline{H'x})=H' \times S_{0}' \ar[d]_{\text{Ad}'} \ar[r]^{\;\quad j'}
                & H' \times \{ N_{0}'+y \} \ar[d]^{\text{Ad}}  \\
   \overline{H'x} \ar[r]_{j}
                &    \mathfrak{h'}      }$$
where $\mbox{Ad}$, $\mbox{Ad}'$ are adjoint actions, and $j'$ and $j$ are the inclusion maps.
 It is easy to know by a direct calculation that the image  of the tangent map of $\mbox{Ad}$ at $(e,y)$ is
$[\mathfrak{h'},y]\oplus N_{0}'$ which is equal to $\mathfrak{h'}$. So $\mbox{Ad}$ is smooth at $(e,y)$. Since the base change preserves smoothness ,  $\mbox{Ad}'$ is smooth at $(e,y)$. So the first part of the lemma follows.

 As $\mbox{char}(\mathbb{K})=p> 2$ is good for $G(V')$ and $G(W)$, it is also good for $H'=G(V')\times G(W)$. So we have $\mathcal{T}_{y}(H' y)=[\mathfrak{h'},y]$, thereby  $\mbox{dim}(H' y)=\mbox{dim}([\mathfrak{h'},y])$.
Combining  with the fact that $\mathcal{T}_{y}(H' y) \bigoplus \mathcal{T}_{y}(y+N_{0}') =[\mathfrak{h'},y]\oplus N_{0}'=\mathfrak{h'}$, we have that $N_{0}'$ and $H' y$ intersect only at $y$ in  a neighborhood $\widehat{O}$ of $y$ in $H' y$, thereby so do $S'$ and $H' y$.  In this neighborhood,  $(\mbox{Ad}')^{-1}(\widehat{O})=\widehat{H'}\times y $ where $\widehat{H'}$ is an open subset of $H' $ such that $\widehat{H'}=\tau^{-1}(\widehat{O})$ with $\tau$ being an orbit map of $y$  under the adjoint  action of $H'$.
The second part of the lemma follows from the following base change
$$\xymatrix{
 \widehat{ H'} \times y \ar[d]_{\Ad''} \ar[r]^{j}
                & H' \times S_{0}' \ar[d]^{\Ad'} \\
  \widehat{O} \ar[r]_{j}
              &         \overline{H'x},     }$$
              and the fact that base change of smooth morphism preserves  the relative dimension (\cite[Proposition 10.1 in Ch.III]{Har}).
 \end{proof}

The foregoing lemma is a modular version of the arguments \cite[12.4]{KP2}. It is still valid  in our situation, essentially due to the separable action of groups (see \cite[17.11.1]{EGA} for general arguments). This is one of the reasons why we  requires  $\mbox{char}(\mathbb{K})=p$ is good.

\subsection{} We are in the position to present the canceling-columns proposition.

\begin{prop} \label{prop3.8} (Induction by canceling columns)
Assume that $\varepsilon$-degeneration $\sigma \leq \eta$ arises from the $\varepsilon'$-degeneration $\sigma' \leq \eta'$ by adding
columns.
Then
$$\text{Sing}(\overline{\mathcal{O}_{\varepsilon,\eta}},\mathcal{O}_{\varepsilon,\sigma})
=\text{Sing}(\overline{\mathcal{O}_{\varepsilon',\eta'}},\mathcal{O}_{\varepsilon',\sigma'}).$$
\end{prop}

 We will prove the above proposition  by the same arguments as in the  proof of \cite[Proposition 13.5]{KP2}. For this, we need to check some details in the case of positive-characteristic fields.

\begin{proof}  It is enough to treat the case when  the $\varepsilon$-degeneration $\sigma \leq \eta$ arises  from the $\varepsilon'$-degeneration $\sigma' \leq \eta'$ by adding a single column. Let $V$ be a quadratic space of type $\varepsilon$ and of  dimension $|\eta|$, $U$ be a quadratic space of type $\ve'$ and of dimension $|\eta'|$. By Lemma \ref{lemma2.6}, we consider the following restrictions of  map $\pi$ and $\rho$ on the  locally closed set
$$
\xymatrix{
 N_{\varepsilon,\eta}:= \pi^{-1}(\overline{\mathcal{O}_{-\varepsilon,\eta^{'}}}) \ar[d]_{\rho} \ar[r]^{\ \ \ \ \  \ \ \ \pi} &   \overline{\mathcal{O}_{-\varepsilon,\eta^{'}}}   \\
\overline{\mathcal{O}_{\varepsilon,\eta}}. }$$
Here $\pi|_{N_{\varepsilon,\eta}}$ ({\sl{resp.}} $\rho|_{N_{\varepsilon,\eta}}$)  is simply written as $\pi$ ({\sl{resp.}} $\rho$).
 We claim that there exists  $x \in N_{\varepsilon,\eta} $ such that $\pi(x)\in \mathcal{O}_{-\varepsilon,\sigma'}$, $\rho(x)\in \mathcal{O}_{-\varepsilon,\sigma}$ and that $x, \pi,\rho $ satisfy the assumption of Definition \ref{def3.1}. The arguments will proceed by steps.

(1) We assert that the map $\pi$ is smooth on the open subset  $$ N_{\ve,\eta}':=N_{\varepsilon,\eta}\cap L'$$
of $ N_{\varepsilon,\eta}$. Indeed,  $\pi (N_{\ve,\eta}')$ is open subset of
$\overline{\mathcal{O}_{-\varepsilon,\eta^{'}}}$ (note that  $\pi (N_{\ve,\eta}')=\pi (L')\cap \overline{\mathcal{O}_{-\varepsilon,\eta^{'}}}$ and $\pi (L')$ is open subset of $\mathfrak{g}(U)$ by Proposition \ref{prop2.7}(1)) and $\pi^{-1}(\pi (N_{\ve,\eta}')) \cap L'= N_{\ve,\eta}'$. Now we have base changes as below:
$$\xymatrix{
  N_{\ve,\eta}'  \ar[d]_{\pi}   \ar[r]^{j}
       &  \ar[d]^{\pi}   L^{'} \\
    \ar[r]_{ \ \ \ \ \ j}
       \pi(N_{\ve,\eta}')  &    \pi( L^{'})        }$$
 where $j$'s  are  inclusions of locally closed subvarieties. As a base change preserves smoothness (cf. \cite[Proposition 10.1 in Ch.III]{Har}) or \cite[(4.9-4.10)]{GW}), we have $\pi$ is smooth on $N_{\ve,\eta}'$ by Proposition \ref{prop2.7}.

(2)  $N_{\ve,\eta}'$ is locally closed $G(U)$-stable subset of $L'$, by lemma \ref{locally closed lemma} we have $\rho|_{N_{\ve,\eta}'}$ is smooth. In another words,  $\rho$ is smooth on  $N_{\ve,\eta}'$ .

(3) By Lemma \ref{lemma2.6}(4), there exists $x \in  N_{\varepsilon,\eta}'$  such that $\pi(x)\in \mathcal{O}_{-\varepsilon,\sigma'}$  and    $\rho(x) \in \mathcal{O}_{\varepsilon,\sigma}$.

Summing up, we complete the proof.
 \end{proof}

\section{Proof of the main result}
\subsection{Proof of Theorem \ref{mainthm}} 
Owing to Theorem \ref{proposition5.2}, we only need to consider all  $\mathcal{O}_{\varepsilon,\sigma}$ of codimention 2 in $\overline{\mathcal{O_{\varepsilon,\eta}}}$, which is  of course a minimal degeneration. By canceling  ``rows" and ``columns" we obtain a minimal irreducible $\varepsilon'$-degeneration $\sigma'\leq \eta'$ (see \cite[\S3.4]{KP2}), which is among types $a,b,c,e$ in Table \ref{table}.
By Lemmas \ref{lemma2.2} and  \ref{lemma2.3}, the normality of  $\mathcal{O}_{\varepsilon',\sigma'}$ at  $ \overline{\mathcal{O_{\varepsilon',\eta'}}}$ is known.
From Theorem \ref{3.3} and lemma \ref{smooth lemma} we can clearly determine as the theorem says, the normality of $\overline{\mathcal{O_{\varepsilon,\eta}}}$ at $\mathcal{O}_{\varepsilon,\sigma}$. The proof is completed.
\begin{remark}
(1) The reason why we exclude type $d$  is that we are not able to prove Proposition 10.2 of \cite{KP2} in positive characteristic. The important tool in \cite{KP2} is Grauert-Riemenschneider vanishing theorem which does not hold ever in positive characteristic. Thus we are not able to obtain the modular version of Proposition 15.4 of \cite{KP2}.

(2) By the arguments in the above proof, one easily has the following observation. If the normality of $\overline{\mathcal{O_{\varepsilon,\eta}}}$ at  $\mathcal{O}_{\varepsilon,\sigma}$ is in type $e$ is known, then we can determine normality of all nilpotent orbits of the orthogonal and symplectic groups. The issue of this question in type $e$  is under investigation.

(3) In comparison with J. Thomsen's results on the normality of nilpotent orbit closures in positive characteristic (cf. \cite{Tho}), our results are more extensively applicable. This is easily seen by some plain examples listed in \S\ref{example} (1) and (2) below.
\end{remark}
\subsection{Examples} \label{example} In the concluding subsection with the same notations and assumptions as in the main theorem,  we illustrate our main result by some examples.
\begin{itemize}
\item[(1)] Let $\mathfrak{g}=\mathfrak{so}_{11}$  and $\eta=[7,2,2])$. Then the nilpotent variety $\overline{\mathcal{O}_{1,\eta}}$ is not normal because
it has codimension 2 degeneration $\mathcal{O}_{1,\sigma}$ with $\sigma=[7,1,1,1,1]$. By erasing first row,  the pair $(\eta,\sigma)$ corresponds to the minimal irreducible degeneration
$([2,2],[1,1,1,1])$ (by erasing first row  ) of type $e$ in Table \ref{table}.
\item[(2)] Let $\mathfrak{g}=\mathfrak{sp}_{8}$ and $\eta=[6,1,1]$. Then the nilpotent variety $\overline{\mathcal{O}_{-1,\eta}}$ is  normal.
$\overline{\mathcal{O}_{-1,\eta}}$ has only one codimension 2 degeneration $\mathcal{O_{-1,\sigma}}$ with $\sigma=[4,2,2]$
and by erasing first column the pair $(\eta,\sigma)$corresponds to minimal irreducible degeneration([5], [3,1,1]) of type $c$ in Table \ref{table}.
\item[(3)] Let $\mathfrak{g}=\mathfrak{sp}_{14}$ and $\eta=[4,4,3,3]$. In this case, we are not able to judge if the nilpotent variety $\overline{\mathcal{O}_{-1,\eta}}$ is  normal. It has a codimension 2 degeneration  of $\mathcal{O}_{-1,\sigma}$ with $\sigma=[4,4,2,2,2]$. This case corresponds to the minimal degeneration of  type $d$ in Table \ref{table}.
\end{itemize}

\section*{Acknowledgements} We thank Hao Chang and Miaofen Chen  for their helpful discussions.
We express our deep thanks to the referee whose suggestions improve our paper. Especially, the referee pointed out some general fact for Lemma \ref{lemma2.3} (see Remark \ref{lemma2.3Rem}(1))
and  provided us with a more  strict proof of Lemma \ref{smooth lemma}.

\begin{thebibliography} {ABCD}
\bibitem{Bro} A. Broer, {\em Line bundles on the cotangent bundle of the flag variety}, Invent. Math. 113 (1991), 1-20.
\bibitem{don} S. Donkin, {\em The normality of closures of conjugacy classes of matrices}, Invent. Math. 101 (1990), 717-736.
\bibitem{EGA} A. Grothendiek and J. Dieudonn\'e, {\em Elements de g\'{e}ometrie alg\'{e}brique} O-\Rmnum{4}, Publ. Math. de
I'I.H.E.S. 11, 20, 24, 32, Paris (1961-67).
\bibitem{Gol} E. Goldstein, {\em Nilpotent Orbits in the symplectic and orthogonal Groups}, Doctoral Theisis, Tufts University, 2011.
\bibitem{GW} U. G$\ddot{o}$rtz and T. Wedhorn, {\em Algebraic geometry I}, Vieweg+Teubner Verlag, Heidelberg, 2010.
\bibitem{Har} R. Hartshorne, {\em Algebraic geometry}, Springer-Verlag, New York, 1977. (GTM 52)


\bibitem{Hes} W. Hesselink, {\em Singularities in the nilpotent scheme of a calssical group}, Trans. Amer. Math. Soc., 222 (1976), 1-32.

\bibitem {Hum} J. E. Humphreys, {\em Linear algebraic groups},  Springer-Verlag, New York, 1975.
(GTM 21)

\bibitem{Jan2} J. C. Jantzen, {\em Nilpotent orbits in representation theory},  Lie theory,  1-211, Progr. Math., 228, Birkh$\ddot{a}$user Boston, Boston, MA, 2004.

\bibitem{KP1} H. Kraft and C. Procesi, {\em Closures of conjugacy classes of matrices are normal}, Invent. Math. 53 (1979),  227-247.
\bibitem{KP2} H. Kraft and C. Procesi, {\em On the geometry of conjugacy classes in classical groups}, Comment. Math. Helv. 57 (1982),  539-602.
\bibitem{Liu} Q. Liu,  {\em Algebraic Geometry and Arithmetic Curves}, Oxford Grad Texts in Math,Oxford University Press 2002.
\bibitem {Mil} J. S. Milne,  {\em $\acute{E}$tale cohomology}, Princeton University Press, Princeton N.J., 1980. (Princiton mathmatical series)

\bibitem{Somm} E. Sommers, {\em Normality of very even nilpotent varieties in $D_{2l}$}, Bull. London Math. Soc.  37 (2005),  351-360.



\bibitem{Tho} J. Thomsen, {\em Normality of certain nilpotent varieties in positive characteristic}, J. Algebra 227 (2000), no. 2, 595-613


\end{thebibliography}
\end{document}